\documentclass{article}
\usepackage{amsmath,amscd}
\usepackage{amssymb}
\usepackage{theorem}
\usepackage{gastex}
\usepackage{url}
\usepackage[dvipsnames]{xcolor}
\newtheorem{theorem}{Theorem}[section]
\newtheorem{proposition}[theorem]{Proposition}
\newtheorem{corollary}[theorem]{Corollary}
\newtheorem{lemma}[theorem]{Lemma}
\usepackage{todonotes}
\usetikzlibrary{shapes,snakes}

\usepackage[colorlinks=true,citecolor=cyan,backref=page]{hyperref}

{\theorembodyfont{\rmfamily}%
  \newtheorem{example}[theorem]{Example}
   }
\newenvironment{proof}{\noindent\textit{Proof.}}
{\QED\vskip\theorempostskipamount} 
\newenvironment{proofof}[1]{\noindent\textit{Proof \protect{#1}.}}
{\QED\vskip\theorempostskipamount}
\def\petitcarre{\vrule height4pt width 4pt depth0pt}
\def\QED{\relax\ifmmode\eqno{\hbox{\petitcarre}}\else{%
  \unskip\nobreak\hfil\penalty50\hskip2em\hbox{}\nobreak\hfil
  \petitcarre
  \parfillskip=0pt \finalhyphendemerits=0\par\smallskip}
  \fi}

\DeclareMathOperator{\Card}{Card}
\DeclareMathOperator{\Hom}{Hom}

\DeclareMathOperator{\Aut}{Aut}
\DeclareMathOperator{\Pal}{Pal}
\def\u(#1){\underline{#1\!}\,}

\def\1{\mathbf{1}}
\newcommand{\F}{FG}
\newcommand{\edge}[1]{\stackrel{#1}{\rightarrow}}

\newcommand{\A}{\mathcal{A}}

\newcommand{\Suf}{\mathcal S}

\numberwithin{equation}{section}

\title{The palindromization map}
\author{Dominique Perrin and Christophe Reutenauer}

\begin{document}

\maketitle
\tableofcontents

\begin{abstract}
  The palindromization map has been defined initially
  by Aldo de Luca in the context of Sturmian words.  It was extended to
  the free group of rank $2$ by Kassel and the second author.
  We extend  their construction to arbitrary alphabets. We also
  investigate the suffix automaton and compact suffix automaton
  of the words obtained by palindromization.
\end{abstract}

\section{Introduction}

The iterated palindromic closure is an injective map mapping
arbitrary words to palindromes. It has been introduced by Aldo
de Luca in~\cite{deLuca1997}.
This map is used to define a representation of Sturmian words
by means of a directive word and is related to a transformation
introduced by Rauzy (see~\cite{Rauzy1984,ArnouxRauzy1991}).
The iterated palindromic closure has been shown in~\cite{KasselReutenauer2008}
to be extendable, in the case of two letters,
to a map (not anymore injective) from the free group into itself
and to have many interesting properties.

In this article, we  study the extension of the iterated palindrome closure
to a
map from the free group on more than two letters to itself. We show that
some of the features appearing with two letters remain valid while
some others do not hold anymore. In particular, we show that the
map is continuous for the profinite topology (Proposition~\ref{propositionProfinite}). We were not able to characterise the kernel of the
map as it is done in~\cite{KasselReutenauer2008}, where it  is
related to the braid group. We also discuss the relation with
noncommutative cohomology evidenced in~\cite{KasselReutenauer2008}
but we show that, on more than two letters, the cocycle
corresponding to the iterated palindromization map is not trivial.

In Section~\ref{sectionSuffixAutomaton}, we
describe the suffix automaton of a word of the form
$\Pal(w)$, that is the minimal automaton of the set of suffixes of this word. We extend to arbitrary alphabets results concerning the suffix automaton; the corresponding results for binary alphabets are from \cite{EpifanioMignosiShallitVenturini2007}.

In Section~\ref{sectionCompactAutomata}, we develop
study compact automata.
These automata have already been studied in the case
of suffix automata (see the chapter
by Maxime Crochemore 
in \cite{Lothaire2005}), but they do not seem to have been
considered before in the general case of automata. We fill this gap and present
a direct definition of a minimal compact automaton,
which is shown to be unique (Corollary~\ref{corollaryMinimalCompact}), together with two other results related to the notion of reduction of automaton (Propositions \ref{reduction} and \ref{elementary}).

This will apply in Section \ref{direct} to the construction of the
minimal compact suffix automaton of $\Pal(u)$.
In that section, we construct directly the minimal compact automaton of the set of suffixes of $\Pal(w)$. The construction 
extends the known construction in the binary alphabet case due to Epifanio, Mignosi, Shallit and Venturini \cite{EpifanioMignosiShallitVenturini2007} (see also \cite{BugeaudReutenauer2022}). It consists in computing
the automaton for $\Pal(ua)$ from the automaton of $\Pal(u)$, by adjoining one state, and several transitions from the first automaton to this state. From this, we deduce the exact form of the automaton (Theorem \ref{Construction}) and several of its properties (Corollaries \ref{only} and \ref{count})

\paragraph{Acknowledgements} We thank Maxime Crochemore for his advice
about compact automata.

\section{The palindromization map}
We denote by $A^*$ the free monoid on the alphabet $A$ and by $1$
 the empty word. We denote by $\tilde{w}=a_n\cdots a_2a_1$
the reversal of the word $w=a_1a_2\cdots a_n$ with $a_i\in A$. The
word $w$ is a  \emph{palindrome} if $\tilde{w}=w$.

We denote by $\F(A)$ the free group on $A$. If $w$ is in $F(A)$ and $a$ in $A$, we 
denote by $|w|_a$ the number of occurrences of $a$ in $w$, where one counts with -1 the occurrences of $a^{-1}$; this is well defined and does not depend on the chosen expression for $w$; we call it the {\it $a$-degree} of $w$. Moreover, define $|w|=\sum_{a\in A}|w|_a$, the {\it algebraic length} of $w$. In particular, if $w\in A^*$, then $|w|$ is the 
{\it length} of $w$.

The {\it reversal} of the 
element
$w=a_1a_2\cdots a_n$ with $a_i\in A\cup A^{-1}$
is the element $\tilde{w}=a_n\cdots a_2a_1$. This does not depend on the chosen expression for $w$.
The map $w\mapsto \tilde{w}$ is an antimorphism, that
is, it satisfies $\widetilde{uv}=\tilde{v}\tilde{u}$.
We also say that an element  $w$ of $\F(A)$ is a palindrome if
$\tilde{w}=w$.

Let us define two morphisms $L\colon u\mapsto L_u$ and $R\colon u\mapsto R_u$
from $\F(A)$ into its group
of automorphisms as follows. For $a,b\in A$, we set
\begin{displaymath}
L_a(b)=\begin{cases}a&\mbox{if $b=a$}\\ab&\mbox{otherwise}\end{cases}
\end{displaymath}
The map $L_a$ is an automorphism of $\F(A)$ since
its inverse is the map 

\begin{displaymath}
L_a^{-1}(b)=\begin{cases}a&\mbox{if $b=a$}\\a^{-1}b&\mbox{otherwise}\end{cases}
\end{displaymath}

Symmetrically, for $a,b\in A$, we set $R_a(b)=\widetilde{L_a(b)}$.

Thus, for example, $L_a(ab^{-1})=aL_a(b)^{-1}=ab^{-1}a^{-1}$
and $R_a(ab^{-1})=aR_a(b)^{-1}=aa^{-1}b^{-1}=b^{-1}$.

Note that $L,R$ are related by two identities. The first one is
\begin{equation}
aR_a(u)=L_a(u)a\label{eqLR1}
\end{equation}
for every $a\in A\cup A^{-1}$ and $u\in\F(A)$. Indeed, this is true
when $u$ is a letter; by rewriting equivalently $R_a(u)=a^{-1}L_a(u)a$, this equality follows since both sides are the image of $u$ under an automorphism of $\F(A)$.

The second one is
\begin{equation}
R_u(v)=\widetilde{L_u(\tilde{v})}\label{eqLR2}
\end{equation}
for every $u,v\in\F(A)$. Indeed, this is true when $u$
is a letter and the general case follows similarly.

Every word $w\in A^*$ is a prefix of some palindrome since
$w\tilde{w}$ is always a palindrome. Thus, there
exists a palindrome of shortest length which has $w$ as a prefix.
Actually, this palindrome is unique.
It is called the \emph{palindromic closure} of $w$
and  denoted $w^{(+)}$. One has
$w^{(+})=yz\tilde{y}$ where $z$ is the longest palindrome
suffix of $w=yz$ (for these results, due to Aldo de Luca, see for example \cite{Reutenauer2019}  Proposition 12.1.1).

For example, $(abaa)^{(+)}=abaaba$, since the longest palindromic suffix of $abaa$ is $z=aa$, and $y=ab$.

Let $\Pal$ be the unique map from $A^*$ to $A^*$ such that $\Pal(1)=1$, and for $w\in A^*$ and $a\in A$
\begin{displaymath}
\Pal(wa)=(\Pal(w)a)^{(+)}.
\end{displaymath}
For a word $w$, $\Pal(w)$ is called
the \emph{iterated palindromic closure} of $w$. 
The iterated palindromic closure has been introduced 
by Aldo de Luca~\cite{deLuca1997}
who has shown that it is injective (see for example \cite{Reutenauer2019} p. 102, where the injectivity is proved by an algorithm).

For example $\Pal(aba)=abaaba$, since $\Pal(a)=a$, $\Pal(ab)=(ab)^{(+)}=aba$, and finally $\Pal(aba)=(abaa)^{(+)}=abaaba$.

The mapping $\Pal$ may be extended to infinite words, since when $u$ is a proper prefix of $v$, $\Pal(u)$ is a proper prefix of $\Pal(v)$.

An important property of $\Pal$ is the following functional equation, known
as \emph{Justin's Formula} \cite{Justin2005}. For all $u,v\in A^*$,
\begin{equation}
\Pal(uv)=\Pal(u)R_u(\Pal(v)).\label{equationJustinR}
\end{equation}
A dual form of \eqref{equationJustinR} (which is the one actually given
in \cite[Lemma 2.1]{Justin2005}) is
\begin{equation}
\Pal(uv)=L_u(\Pal(v))\Pal(u).\label{equationJustinL}
\end{equation}
Indeed, assuming \eqref{equationJustinL}, we have using \eqref{eqLR2},
\begin{eqnarray}
\Pal(uv)&=&\widetilde{\Pal(uv)}=\widetilde{\Pal(u)}\widetilde{L_u(\Pal(v))}\\
&=&\Pal(u)R_u(\widetilde{\Pal(v)})=\Pal(u)R_u(\Pal(v)).\label{eqEquiv}
\end{eqnarray}
Hence \eqref{equationJustinL} implies \eqref{equationJustinR}. The reverse implication is true, too, as is similarly verified.

We want to prove the following result, extending the construction of~\cite{KasselReutenauer2008}
to arbitrary alphabets.
\begin{theorem}
There exists a unique extension of $\Pal$ to a map
from $\F(A)$ to itself fixing every $a\in A\cup A^{-1}$
and satisfying~\eqref{equationJustinR}
for every $u,v\in\F(A)$.
\end{theorem}

We
will still denote by $\Pal$
the extension of the iterated palindromic closure
to the free group and call it the \emph{palindromization map}.
We will see below that the extension also satisfies \eqref{equationJustinL}, for any $u,v\in \F(A)$.

The statement follows directly from the following property.
\begin{proposition}
Let $\alpha:u\in \F(A)\mapsto \alpha_u\in\Aut(\F(A))$ be a morphism from the free group on $A$
to the group of its automorphisms, such that $\alpha_a(a)=a$
for every $a\in A$. There exists a unique map
$f:\F(A)\to\F(A)$ such that
\begin{itemize}
\item[\rm (i)] $f(a)=a$ for every $a\in A\cup A^{-1}$.
\item[\rm (ii)] 
\begin{equation}f(uv)=f(u)\alpha_u(f(v))\label{eqSequential}
\end{equation} for every $u,v\in\F(A)$.
\end{itemize}

\end{proposition}
\begin{proof}
We prove first uniqueness of $f$. Equation~\eqref{eqSequential} with $u=v=1$ implies
that $f(1)=1$.

Next, the same equation implies
$$
f(au)=a\alpha_a(f(u)).
$$
for each $u\in\F(A)$ and each $a\in A\cup A^{-1}$. Thus 
there is a unique  $f$ such that
for each reduced word $au$ with $a\in A\cup A^{-1}$
one has $f(au)=a\alpha_a(f(u))$. Indeed, this
is true when $u=1$ and follows easily
using induction on the length of $au$.
Thus there is at most one map $f$ satisfying conditions
(i) and (ii).

To prove the existence, let us
prove that, for this map $f$, Equation~\eqref{eqSequential}
holds for every $v\in\F(A)$ by induction on the length $l(u)$ of
the reduced word $u$. It holds
for $|u|=0$. Next, let $u\in\F(A)$ be of positive length.
Set $u=u'a$ with $a\in A\cup A^{-1}$, in reduced form. Then
$l(u')=l(u)-1$. Since $uv=u'(av)$, the induction hypothesis implies
\begin{eqnarray*}
f(uv)&=&f(u')\alpha_{u'}(f(av)).
\end{eqnarray*}

Applying again the induction hypothesis, we obtain
$$f(u)=f(u'a)=f(u')\alpha_{u'}(a).$$

- Assume first that $av$ is reduced. Then, by definition of $f$, we have
$f(av)=a\alpha_a(f(v))$. Thus
\begin{eqnarray*}
f(uv)&=&f(u')\alpha_{u'}(a\alpha_a(f(v)))\\
&=&f(u')\alpha_{u'}(a)\alpha_{u}(f(v)).
\end{eqnarray*}

Thus we obtain
\begin{eqnarray*}
f(uv)
&=&f(u)\alpha_u(f(v)).
\end{eqnarray*}

- Assume now that $av$ is not reduced; set $v=a^{-1}v'$ in reduced form. Then, by definition of $f$, we have $f(v)=f(a^{-1}
v')=a^{-1}\alpha_{a^{-1}}(f(v'))$.

Thus, since
$\alpha_a(a^{-1})=a^{-1}$, we have
\begin{eqnarray*}
f(u)\alpha_u(f(v))&=&f(u')\alpha_{u'}(a)\alpha_{u'a}(a^{-1}\alpha_{a^{-1}}(f(v'))\\
&=&f(u')\alpha_{u'}(a)\alpha_{u'}(a^{-1})\alpha_{u'}(f(v'))\\
&=&f(u')\alpha_{u'}(f(v'))
\end{eqnarray*}
which by induction hypothesis is equal to $f(u'v')=f(uv)$.
\end{proof}

We now verify the following property of the palindromization map.

\begin{proposition}\label{propositionPalisPal}
For every $w\in\F(A)$, $\Pal(w)$ is a palindrome.
\end{proposition}
\begin{proof}
For every $u\in\F(A)$ and $a\in A\cup A^{-1}$, we have
\begin{equation}
aR_a(\tilde{u})=\widetilde{R_a(u)}a.\label{eqaR_a}
\end{equation}
Indeed, $aR_a(\tilde{u})=L_a(\tilde{u})a=\widetilde{R_a(u)}a$ by \eqref{eqLR1} and \eqref{eqLR2}. It follows from \eqref{eqaR_a} that for every $a\in A\cup A^{-1}$
and every palindrome $u\in\F(A)$, $aR_a(u)$ is a palindrome.

Let us now show by induction on the length of the reduced word
representing $w\in\F(A)$ that $\Pal(w)$ is a palindrome.
It is true if $w=1$. Next, set $w=au$ in reduced form
with $a\in A\cup A^{-1}$
and $u\in \F(A)$. We have by \eqref{equationJustinR},
$\Pal(w)=aR_a(\Pal(u))$. By induction hypothesis, $\Pal(u)$
is a palindrome and it follows from the previous observation
that $\Pal(w)$ is a palindrome, too.
\end{proof}
It follows from Proposition \ref{propositionPalisPal},
using the same argument as in \eqref{eqEquiv},
that the map $\Pal$ satisfies also \eqref{equationJustinL}
for every $u,v\in\F(A)$.

As an example, we have $Pal(b^{-1})=b^{-1}$ and $\Pal(ab^{-1})=\Pal(a)R_a(Pal(b^{-1}))=aR_a(b^{-1})=a(ba)^{-1}=
b^{-1}$. This shows
that the extension of $\Pal$ to $\F(A)$ is not injective.
In the case of a binary alphabet, one can characterize
the kernel of $\Pal$ as follows. 

Let $B_3$ be the braid group on three strands defined as
\begin{displaymath}
B_3=\langle \sigma_1,\sigma_2\mid \sigma_1\sigma_2\sigma_1=\sigma_2\sigma_1\sigma_2\rangle.
\end{displaymath}
Let $\beta:\F(a,b)\to B_3$ be the morphism $\beta:a\mapsto \sigma_1,b\mapsto\sigma_2$.
For any $u,v\in \F(a,b)$,
one has
$\Pal(u)=\Pal(v)$ if and only if $\beta(u^{-1}v)\in\langle \sigma_1\sigma_2^{-1}\sigma_1^{-1}\rangle$ (see~\cite[Proposition 5.2]{KasselReutenauer2008}).
No such characterization is known on more than two letters.
\section{Semidirect products, cocycles and sequential functions}
We discuss now several interpretations of Justin's Formula.
\paragraph{Semidirect products}
Observe first that Equation~\eqref{eqSequential}
(and thus also Equation~\eqref{equationJustinR}) can be expressed in terms
of semidirect products. Indeed, consider the semidirect product
$\F(A)*_\alpha\F(A)$
of $\F(A)$ with itself corresponding to the morphism $\alpha$
from $\F(A)$ into $\Aut(\F(A))$. By definition, it is the set
of pairs $(u,v)\in\F(A)\times\F(A)$ with the product
\begin{displaymath}
(u,v)(r,s)=(u\alpha_v(r),vs)
\end{displaymath}
Equation~\eqref{eqSequential} expresses the fact that $\delta:w\mapsto(f(w),w)$
is a morphism from $\F(A)$ to $\F(A)*_\alpha\F(a)$. Indeed, assuming
\eqref{eqSequential}, we have for every $u,v\in\F(A)$,
\begin{displaymath}
\delta(u)\delta(v)=(f(u),u)(f(v),v)=(f(u)\alpha_u(f(v)),uv)=f(uv),uv)=\delta(uv).
\end{displaymath}
This proves the following statement.
\begin{proposition}
The map $u\mapsto (\Pal(u),u)$ is the unique morphism
from $\F(A)$ to $\F(A)*_R\F(A)$ sending every $a\in A$ to $(a,a)$.
\end{proposition}
\paragraph{Cocycles}
Justin's Formula is also related to the notion of nonabelian
group cohomology, as pointed out in~\cite{KasselReutenauer2008}. A
 function $f$, from $\F(A)$
to itself, is a 1-\emph{cocycle}, with respect to a
group morphism $\alpha:u\mapsto \alpha_u$ from $\F(A)$ to $\Aut(\F(A))$,
if \eqref{eqSequential} holds for all $u,v\in\F(A)$.
Thus $\Pal$, as a function from $\F(A)$ to itself, is
a 1-cocycle with respect to the morphism $R$.
Such a 1-cocycle is \emph{trivial} if there is an
element $x\in\F(A)$ such that 
\begin{equation}
f(u)=x^{-1}\alpha_u(x)\label{eqTrivial}
\end{equation}
for every $u\in \F(A)$. When $A$ has two elements, one has
$\Pal(u)=(ab)^{-1}R_u(ab)$
by \cite[Equation (3.1)]{KasselReutenauer2008}. Thus the 1-cocycle $\Pal$
is trivial. This is not the case on more than two letters.
Indeed, suppose that $x\in \F(A)$ is such that
$\Pal(u)=x^{-1}R_u(x)$ for all $u\in\F(A)$. One has then
$xa=R_a(x)$ for every $a\in A$ and thus, by taking the $a$-degree:
$|x|_a+1=|x|$. This implies, by summing over all $a\in A$, $|x|(\Card(A)-1)=\Card(A)$,
which is impossible for $\Card(A)\ge 3$ since $|x|$ is an integer.

\paragraph{Sequential functions}
Equation~\eqref{equationJustinR} can also be seen as expressing that,
as a function from $A^*$ to $A^*$, the map
$\Pal$ is a \emph{sequential function}, that is, a function computed by
a \emph{sequential transducer}. Let us recall the definition
of a sequential transducer on a set $Q$ of states. Let
\begin{displaymath}
(q,a)\in Q\times A\mapsto q\cdot a\in Q
\end{displaymath}
be a map, called the \emph{transition function}. This map extends to a right action of $A^*$ on $Q$
by $q\cdot (ua)=(q\cdot u)\cdot a$ for $u\in A^*$ and $a\in A$.
In addition, let
\begin{displaymath}
(q,a)\in Q\times A\mapsto q* a\in A^*
\end{displaymath}
be a map called the \emph{output function}. This map extends to a map from $Q\times A^*$ to $A^*$ by
\begin{displaymath}
q*(ua)=(q*u)((q\cdot u)*a).
\end{displaymath}
Given a sequential transducer on $Q$ defined by the maps $(q,a)\mapsto q\cdot a$ and $(q,a)\mapsto q*a$ and given an initial state
$i\in Q$, the function $f:A^*\to A^*$ defined by the
transducer is 
\begin{displaymath}
f(w)=i*w
\end{displaymath}
\begin{proposition}\label{propositionSequential}
  The function $\Pal$ is defined by the transducer on the set of
  states $\Aut(\F(A))$
with transition and output functions 
\begin{displaymath}
R_u\cdot a=R_{ua}\mbox{ and } R_u*a=R_u(a)
\end{displaymath}
respectively, and with initial state $i=R_1$.
\end{proposition}
\begin{proof}
We prove by induction on the length of $w\in A^*$ that $\Pal(w)=i*w$.
It is true for $w=1$ and next, assuming that $\Pal(u)=i*u$,
we obtain for every $a\in A$,
\begin{displaymath}
i*(ua)=(i*u)((i\cdot u)*a)=\Pal(u)R_u(a)=\Pal(ua).
\end{displaymath}
\end{proof}

\section{Uniform continuity of $\Pal$ for the profinite distance}
The {\it profinite topology} on the free group $\F(A)$ is the topology
generated by the inverse images of subsets of a finite
group $F$ by a morphism $\varphi:\F(A)\to F$. Equivalently, it is
the coarsest topology such that every morphism to a finite discrete group is
continuous. This topology on the free group
was introduced by Hall  (see~\cite{Hall1950}).
It is a particular case of a more general notion 
which extends to varieties of groups and
also of semigroups (see~\cite{RibesZalesskii2010}
and~\cite{Almeida2005}).

The following is proved in~\cite{KasselReutenauer2008} for the case of a binary alphabet.

\begin{proposition}\label{propositionProfinite}
The map $\Pal:\F(A)\to \F(A)$ is continuous for the profinite topology.
\end{proposition}

Actually we shall prove a slightly more general result. Following \cite{Almeida2005} p. 57, we define the {\it profinite distance} $d$ 
on $FG(A)$ by the formula 
$$
d(u,v)=\frac{1}{r(u,v)}
$$
where $u,v$ are distinct, and where $r(u,v)$ is the minimal cardinality of a group $G$ such that, for some morphism $\varphi:FG(A)
\to G$, $\varphi(u)\neq \varphi(v)$. Since $FG(A)$ is residually finite, $r(u,v)$ is a finite integer for each pair $u,v$ of distinct 
elements in $FG(A)$. It is easy to prove that $r(u,w)\geq \min(r(u,v),r(v,w))$, hence
$$
d(u,w)\leq \max(d(u,v),d(v,w)),
$$
which means that $d$ is an ultrametric distance.

The topology of $FG(A)$ induced by $d$ is precisely the profinite topology.

%
%

Proposition \ref{propositionProfinite} follows from the next result, since each uniformly continuous function is continuous.

\begin{proposition}\label{propositionProfiniteU}
  The map $\Pal:\F(A)\to \F(A)$
  is uniformly continuous for the profinite distance.
\end{proposition}

We need in the proof the following characterizations of uniformly continuous functions.

1. A function $P:FG(A)\to FG(A)$ is uniformly continuous if and only if for any morphism $\varphi: FG(A)\to G$, the function $\varphi\circ P:FG(A)\to G$ is uniformly continuous, where $G$ has the discrete distance.

2. A function $P':\F(A)\to G$, with $G$ a finite group with discrete distance, is uniformly continuous if and only if it factorizes as $P'=h\circ 
\psi$, where $\psi:FG(A)\to S$ is a morphism into a finite group, and $h:S\to G$ is some function.

\begin{proof} We apply the first criterion above to the function $P=\Pal$. Thus 
let $\varphi:\F(A)\to G$ be a group morphism into a finite group $G$. 

We define a right action of $\F(A)$ on the finite set $Q=G\times \Hom(\F(A),G)$
by
\begin{equation}
(g,\pi)\cdot w=(g\pi(\Pal(w)),\pi\circ R_w).\label{eqProfinite}
\end{equation}

It is indeed a right action, since the associativity follows from
\begin{eqnarray*}
((g,\pi)\cdot u)\cdot v&=&(g\pi(\Pal(u)),\pi\circ R_u)\cdot v\\
&=&(g\pi(\Pal(u))\pi\circ R_u(\Pal(v)),\pi\circ R_u\circ R_v)\\
&=&(g\pi(\Pal(u)R_u(\Pal(v))),\pi\circ R_{uv})\\
&=&(g\pi(\Pal(uv)),\pi\circ R_{uv})=(g,\pi)\cdot (uv).
\end{eqnarray*}
It follows from Equation~\eqref{eqProfinite} that for every $w\in \F(A)$,
one has
\begin{displaymath}
(1_G,\varphi)\cdot w=(\varphi(\Pal(w)),\varphi\circ R_w).
\end{displaymath}

In order to apply the second criterion, we define now a morphism $\psi:\F(A)\to S$, where $S$ is the symmetric group
on $Q$:
for any $w$ in $\F(A)$, $\psi(w)$ is the permutation $q\mapsto q\cdot w$. Then $\psi$ is a morphism, because of the right action defined above, letting $S$ act on the right on $Q$. Note that $S$ is a finite group, since $Q$ is finite.

Define a function $h:S\to G$; it sends each element $\sigma\in S$ onto the first component $g$ of the pair $(g,\pi)=(1_G,\varphi)\sigma$.
Thus we have $h\circ\psi(w)=$ first component of $(1_G,\varphi)\cdot w$, which is equal to $\varphi\circ \Pal(w)$; thus $h\circ \psi=\varphi\circ \Pal$, which allows to conclude, according to the second criterion, that $\varphi\circ\Pal$ is uniformly continuous. With the first criterion, we see that $\Pal$ is uniformly continuous.
%
\end{proof}
Since $A^*$ is dense in $\F(A)$ for the profinite topology
(see~\cite{Almeida2005} or \cite{AlmeidaCostaKyriakoglouPerrin2020} for example),
we deduce from Proposition~\ref{propositionProfinite} the following statement.
\begin{corollary}
The map $\Pal:\F(A)\to\F(A)$ is the unique continuous extension to $\F(A)$
of the iterated palindromic closure of $A^*$.
\end{corollary}

Though $\Pal$ is continuous for the profinite topology, it is not continuous for the pro-$p$ topology on the free group $F_2$ of rank $2$, where $p$ is a prime number.
Recall that the pro-$p$ topology is the coarsest topology
such that every group homomorphism from $F_2$ into a finite $p$-group
is continuous (see \cite[Remark 6.3]{KasselReutenauer2008}).

\section{Suffix automaton}\label{sectionSuffixAutomaton}

The minimal automaton of the set of suffixes of a word $w$
is called the \emph{suffix automaton} of $w$. These
automata has been extensively studied (see~\cite[Chapter 2]{Lothaire2005})
A striking property (originally due to \cite{Hausler1985})
is that its number of states is at most $2|w|-1$.

Let $u$ be a word over an arbitrary alphabet $A$.
We denote by $\Suf(u)$ the suffix automaton of
 $\Pal(u)$.
 
Part (i) of the following result is not new, but we give a proof for sake of completeness.

\begin{theorem}\label{theoremPalSuf}
  The automaton $\Suf(u)$  has the following properties:
  \begin{enumerate}
  \item[\rm(i)] It has $|\Pal(u)|+1$ states, which may be naturally identified
    with the prefixes of $\Pal(u)$.
  \item[\rm(ii)] Its terminal states are the palindromic prefixes of $\Pal(u)$.
    
    \end{enumerate}
  
\end{theorem}

This result, for a binary alphabet, is due to
~\cite{EpifanioMignosiShallitVenturini2007}
((i) in the binary case also follows from Theorem 1 in \cite{SciortinoZamboni2007}, which characterizes the binary words whose suffix automaton has $|w|+1$ states); see also \cite{EpifanioMignosiShallitVenturini2007},
\cite{EpifanioFrougnyGabrieleMignosiShallit2012},
\cite{BugeaudReutenauer2022}.
Moreover, for general alphabets, Part (i) of the theorem is a consequence of the remark in Section 5 in Fici's article \cite{Fici2011}. Note that 
that characterizations of the words $w$ such that the suffix automaton of $w$ has $|w|+1$ states have been given by Fici \cite{Fici2011} and Richomme \cite{Richomme2017}.

A factor $w$ of a word $u$ (resp. an infinite word $s$) is \emph{left-special} 
if there are at least two letters $a$ such that $aw$ is a factor of $u$
(resp. $s$).

The following is from \cite[Proposition 5]{DroubayJustinPirillo2001}
(see also~\cite[Proposition 1.5.11]{DurandPerrin2021}). Recall that for any infinite word $x$, the infinite word $\Pal(x)$ is well-defined. 

\begin{proposition}\label{propositionDJP}
  If $s=\Pal(x)$ for some infinite word $x$, the left-special
  factors of $s$ 
  are the prefixes of $s$.
\end{proposition}
We will use the following consequence. 

\begin{corollary}\label{corollaryLS}
  For any (finite) word $u$, the left special factors of $\Pal(u)$
  are prefixes of $\Pal(u)$.
\end{corollary}
\begin{proof}
  Set $s=\Pal(u^\omega)$. Let $p$ be a left-special factor of $\Pal(u)$.
   Since $\Pal(u)$ is a prefix of
  $s$, the word $p$ is also left-special with respect to $s$.
  By Proposition~\ref{propositionDJP}, $p$ is a prefix of $s$
  and thus of $\Pal(u)$.
\end{proof}

Note that not all prefixes of $\Pal(u)$ need to be left-special.
For example, if $u=ab$, the factor $ab$ is not a left-special
factor of $\Pal(u)=aba$.

Given a language $L$, a \emph{residual} of $L$ is
set of the form $u^{-1}L=\{v\in A^*\mid uv\in L\}$.
It is well-known that the minimal automaton of a language $L$,
denoted $\A(L)$, has the set of nonempty residuals of $L$ as set of states.

\begin{proofof}{of Theorem~\ref{theoremPalSuf}}
   Let $1$ be the
  initial
  state of $\Suf(u)$. Let $P$ be the set of prefixes of $\Pal(u)$. The
  map $\alpha\colon p\mapsto 1\cdot p$ is injective. Indeed,
  let $p,p'\in P$ be such that $1\cdot p=1\cdot p'$.
  Assuming that $|p|\le|p'|$, let $r$ be such that $p'=pr$.
  Then $1\cdot pr=1\cdot p'=1\cdot p$ and thus $(1\cdot p)\cdot r=1\cdot p$.
  Since the language recognized by $\Suf(u)$ is finite,
  the graph of $\Suf(u)$ is acyclic, which forces $r=1$.
  Thus $p=p'$.

  Let us show now that $\alpha$ is surjective. Let
  $q$ be a state of the automaton $\Suf(u)$.
  Let
  $w$ be a word such that $1\cdot w=q$. Since the automaton is co-accessible, there is some word $s$ such that
  $1\cdot ws$ is a terminal state, and thus
  $ws$ is a suffix of $\Pal(u)$. Hence the word $w$ is a factor of $\Pal(u)$.
  Let $p$ be the shortest prefix of $\Pal(u)$ such that
  $pw$ is a prefix of $\Pal(u)$. Let us show by induction
  on the length of $p'$
  that for every suffix $p'$ of $p$, one has $1\cdot p'w=q$.
  It is true if $p'=1$. Otherwise, set $p'=ap''$.
  We have $1\cdot p''w=q$ by induction hypothesis.

  Let $s$ be an arbitrary word such that $ws$ belongs to the
  set $S$ of suffixes of $\Pal(u)$. Since $1\cdot p''w=1\cdot w$,
  and since the automaton $\Suf(u)$ is the minimal automaton
  of $S$, so that $(p''w)^{-1}S=w^{-1}S$, and $s\in w^{-1}S$, 
  we have $p''ws\in S$. We cannot have $p''ws=\Pal(u)$, since otherwise
  $p''w$ is a prefix of $\Pal(u)$ with $p''$ shorter than $p$.

  Thus there is some $b\in A$  such that
  $bp''ws\in S$ . Since $ap''w=p'w$ is a factor of $\Pal(u)$
  and since, by Corollary~\ref{corollaryLS}, $p''w$ is not left-special
  (because $p''w$ is not a prefix of $\Pal(u)$ for the same reason as above),
  we have $a=b$.
  This shows that $ws\in S\Rightarrow p'ws\in S$. The converse is
  also true and thus that $w^{-1}S=(p'w)^{-1}S$, that is $1\cdot p'w=1\cdot w=q$,
  since the automaton is minimal.

  We conclude that $1\cdot pw=q$. This shows that $\alpha$
  is surjective and proves property (i).
  Next, if $p$ is a prefix of $\Pal(u)$ which is also a suffix,
  then $p$ is a palindrome and thus property (ii) is true.
\end{proofof}

Note that the automaton $\Suf(u)$ has the additional property
\begin{enumerate}
\item[\rm(iii)] The label of an edge depends only on its end.
\end{enumerate}
Actually, this property holds for any suffix automaton, as is well-known. Indeed,
if $p,q$ are two states of the suffix automaton
of a word $w$ such that $p\cdot a=q\cdot b=r$, let $u,v,t$
be such that $1\edge{u}p\edge{a}r\edge{t}s$ and
$1\edge{v}q\edge{b}r\edge{t}s$ with $s$ a terminal state. Then $uat$ and $vbt$
are suffixes of $w$, which implies $a=b$.

\begin{example}
  Consider $u=abc$. We have $\Pal(u)=abacaba$ and the automaton $\Suf(abc)$
  is represented in Figure~\ref{figureS(abc)}.
  \begin{figure}[hbt]
    \centering
      \tikzset{node/.style={circle,draw,minimum size=0.4cm,inner sep=0.4pt}}
  \tikzset{box/.style={draw,minimum size=0.4cm,inner sep=0pt}}
  \tikzstyle{loop right}=[in=-40,out=40,loop]
  \begin{tikzpicture}
    \node[node,double](1)at(0,0){$0$};\draw[->](-.5,0)--node{}(1);
    \node[node,double](a)at(1.5,0){$1$};
    \node[node](ab)at(3,0){$2$};
    \node[node,double](aba)at(4.5,0){$3$};
    \node[node](abac)at(6,0){$4$};
    \node[node](abaca)at(7.5,0){$5$};
    \node[node](abacab)at(9,0){$6$};
    \node[node,double](abacaba)at(10.5,0){$7$};

    \draw[->,above](1)edge node{$a$}(a);
    \draw[->,above](a)edge node{$b$}(ab);
    \draw[->,above](ab)edge node{$a$}(aba);
    \draw[->,above](aba)edge node{$c$}(abac);
    \draw[->,above](abac)edge node{$a$}(abaca);
    \draw[->,above](abaca)edge node{$b$}(abacab);
    \draw[->,above](abacab)edge node{$a$}(abacaba);
    \draw[bend left=40,above,->](1)edge node{$b$}(ab);
    \draw[bend left=40,above,->](a)edge node{$c$}(abac);
    \draw[bend right=40,above,->](1)edge node{$c$}(abac);
    \end{tikzpicture}
  \caption{The automaton $\Suf(abc)$.}\label{figureS(abc)}
  \end{figure}
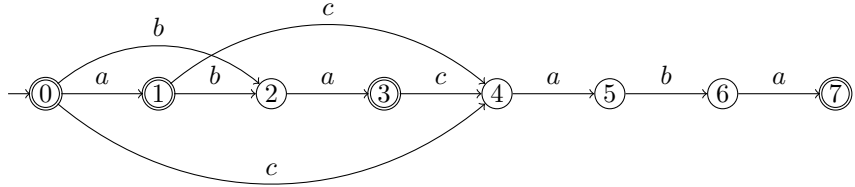
\end{example}

\section{Compact automata}\label{sectionCompactAutomata}
We explore the notion of compact automaton in which the
edges can be labeled by nonempty words instead of letters.
This version of automata appears, in the case of compact suffix automata,
in~\cite{BlumerEhrenfeuchtHaussler1989} or
\cite{CrochemoreVerin1997}. It is also presented in the chapter
by Maxime Crochemore 
in \cite{Lothaire2005}. In particular, the construction
of a minimal compact suffix automaton is described
(see also \cite{BlumerBlumerHausslerMcConnellEhrenfeucht1987}).
We will show here that it is possible to define in complete
generality a minimal compact automaton for every language.

A \emph{compact automaton} $\A=(Q,E,I,T)$ is given by a set of states $Q$,
a set of edges $E\subset Q\times A^+\times Q$,
a set initial states $I\subset Q$ and a set of terminal states $T\subset Q$.
A path $p_0\edge{u_0}p_1\edge{u_1}p_2\ldots \edge{u_{n-1}} p_n$ is a sequence
of consecutive edges. Its label is $u_0u_1\cdots u_{n-1}$.
The language recognized by $\A$, denoted $L(\A)$, is the set of labels
od \emph{successful} paths, that is
 paths from $I$
to $T$.

An ordinary automaton is clearly a particular case of a compact automaton.

A compact automaton is \emph{deterministic} if $\Card(I)=1$ and if for
every state $p$, the labels of the edges starting at $p$
begin with distinct letters.

Again, an ordinary deterministic automaton is deterministic
as a compact automaton.

\begin{example}
  The compact automaton of Figure~\ref{figureCompact1}
  is deterministic. Its initial state (indicated with an incoming arrow)
  is $0$ and $2$ (with a double circle) is the unique terminal state.
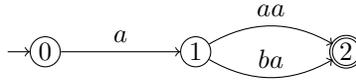
\begin{figure}[hbt]
  \centering
  \tikzset{node/.style={circle,draw,minimum size=0.4cm,inner sep=0.4pt}}
  \tikzset{box/.style={draw,minimum size=0.4cm,inner sep=0pt}}
  \tikzstyle{loop right}=[in=-40,out=40,loop]
  \begin{tikzpicture}
    \node[node](0)at(0,0){$0$};\draw[->](-.5,0)--node{}(0);
    \node[node](1)at(2,0){$1$};
    \node[node,double](2)at(4,0){$2$};

    \draw[->,above](0)edge node{$a$}(1);
    \draw[->,above,bend left](1)edge node{$aa$}(2);
    \draw[->,above,bend right](1)edge node{$ba$}(2);
    \end{tikzpicture}
\caption{A deterministic compact automaton.}\label{figureCompact1}
\end{figure}
\end{example}

The set of \emph{special} states of a compact automaton
$\A=(Q,E,I,T)$
is the set $Q_s$ of states $q$ which either belong to $I\cup T$ or such that
there are edges going out of $q$ with labels beginning with distinct
letters.

Let $p,q$ be special states.
A path $p\edge{w}q$ is \emph{special} if the only special states on
  the path are its origin and its end.

A \emph{reduction} from a deterministic compact automaton $\A=(Q,E,i,T)$ onto
a deterministic compact automaton $\A'=(Q',E',i',T')$ is a map $\varphi$
from $Q_s$ onto $Q'_s$
such that
\begin{enumerate}
\item $\varphi(i)=i'$,
\item $\varphi(p)\in T'$ if and only if $p\in T$,
\item for every $p,q\in Q_s$, there is a special
  path $p\edge{w}q$ in $\A$ if and only
    there is a special path $\varphi(p)\edge{w}\varphi(q)$ in $\A'$.
\end{enumerate}
An automaton is \emph{trim} if every state is on some successful path.
\begin{proposition}\label{same}
If $\A,\A'$ are deterministic compact automata
and if $\varphi\colon \A\to\A'$ is a reduction, then $L(\A)=L(\A')$.
\end{proposition}
\begin{proof}
If $w$ is in $L(\A)$, there is a path $i\edge{w}t$
with $t\in T$. Let $w=w_0w_1\cdots w_n$ be the factorisation
of $w$ such that the path has the form
$q_0\edge{w_0}q_1\edge{w_1}\cdots q_n\edge{w_n}q_{n+1}$
with each path $q_i\edge{w_i}q_{i+1}$ being special and where
$q_0=i,q_{n+1}=t$. Since $\varphi$ is a reduction,
there is for each $i$
with $0\le i\le n$, a special path $\varphi(q_i)\edge{w_i}\varphi(q_{i+1})$.
Thus there is in $\A'$ a path $i'=\varphi(i)\edge{w}\varphi(t)\in T'$,
which implies that $w$ is in $L(\A')$. 

Conversely, let $w\in L(\A')$. Then there exists a path $i'\edge{w}t'$
with $t'\in T'$. We may decompose this path as $i'=q'_0\edge{w_0}q'_1\edge{w_1}\cdots q'_n\edge{w_n}q'_{n+1}=t'$, where each 
path $q'_i\edge{w_i}q'_{i+1}$ is special.
By surjectivity of $\varphi$, we have $q'_i=\varphi(q_i)$, $q_0=i$ by Condition 1 in the definition 
of reduction, and $q_{n+1}=t\in T$ by Condition 2.
Next, there is in $\A$ a special path $q_i\edge{w_i}q_{i+1}$ by Condition 3. Thus there is in $\A$ a path $i\edge{w} t$ and consequently
$w\in L(\A)$.
\end{proof}

Given a language $L\subset A^*$, a nonempty residual $u^{-1}L$ is called
\emph{special} if either
\begin{enumerate}
  \item
    $u=1$, or
  \item $u\in L$, or
  \item there are two $v,w\in u^{-1}L$ which begin by different letters.
\end{enumerate}

The \emph{minimal compact automaton} of a language $L$, denoted $\A_c(L)$
is the following compact automaton. The set of states is the
set of special residuals of $L$.
The initial state is $L$ and the terminal states are the $u^{-1}L$
such that $u$ is in $L$. The edges are the $(p,v,q)$
such that $p=u^{-1}L$, $q=(uv)^{-1}L$ and there is no
factorization $v=v'v''$ with $v',v''$ nonempty such that $(uv')^{-1}L$
is a special residual.
By definition, $\A_c(L)$ is deterministic and all its states  are special; in particular, all its special paths are edges.
\begin{example}
  The compact automaton of Figure~\ref{figureCompact1}
  is the minimal compact automaton of the language $\{aaa,aba\}$.
  \end{example}
\begin{example}
  The compact automaton of Figure~\ref{figureCompact2} is deterministic.
  Its initial state is indicated by an incoming arrow, and all states are terminal.
\begin{figure}[hbt]
  \centering
  \tikzset{node/.style={circle,draw,minimum size=0.4cm,inner sep=0.4pt}}
  \tikzset{box/.style={draw,minimum size=0.4cm,inner sep=0pt}}
  \tikzstyle{loop right}=[in=-40,out=40,loop]
  \begin{tikzpicture}
\node[node,double](1)at(0,0){};\draw[->](-.5,0)--node{}(1);
    \node[node,double](a)at(2,0){};
    \node[node,double](aba)at(4,0){};
     \node[node,double](abacaba)at(6,0){};

    \draw[->,above](1)edge node{$a$}(a);
    \draw[->,above](a)edge node{$ba$}(aba);
    
    \draw[->,above](aba)edge node{$caba$}(abacaba);
   
    \draw[bend left=40,above,->](1)edge node{$ba$}(aba);
    \draw[bend left=40,above,->](a)edge node{$caba$}(abacaba);
    \draw[bend right=40,above,->](1)edge node{$caba$}(abacaba);
    \end{tikzpicture}
\caption{A deterministic compact automaton.}\label{figureCompact2}
\end{figure}
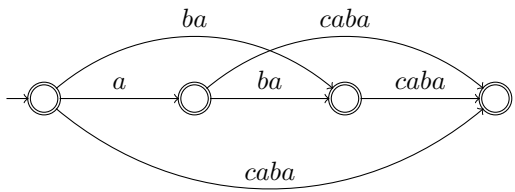
This automaton is the minimal compact automaton of the
set of suffixes of the word $abacaba=\Pal(abc)$.
\end{example}
\begin{proposition}\label{reduction} For every trim deterministic compact automaton $\A$, there is
a unique reduction from $\A$ onto the minimal compact automaton
of $L(\A)$.
\end{proposition}
\begin{proof}
  Let $\A=(Q,i,T)$ and $L=L(\A)$. Set $\A_c(L)=(R,j,S)$. We define
  a mapping $\varphi\colon Q_s\to R$ as follows. First $\varphi(i)=j$, so that Condition 1 in the definition of reduction is satisfied.
  Next, for $p\in Q_s$ let $u$ be such that $i\edge{u}p$.
  We set $\varphi(p)=u^{-1}L$. The map is well-defined because
  if $i\edge{u}p$ and $i\edge{u'}p$, then $u^{-1}L=u'^{-1}L$.
  If $p$ is in $T$, then $u$ is in $L$ and thus $\varphi(p)$
  is in $S$; conversely, if $\varphi(p)\in S$, then $u\in L$, hence $p\in T$, and Condition 2 is satisfied.
  
  If there are two edges $p\edge{av}q$ and $p\edge{a'v'}q'$
  in $\A$ with $a\ne a'$, let $w,w'$ be such that $q\edge{w}t$,
  $q'\edge{w'}t'$ and $t,t'\in T$. Then $avw,a'v'w'\in u^{-1}L$
  and thus $\varphi(p)$ is a special residual. This shows
  that $\varphi$ maps $Q_s$ into $R_s$.

  The mapping $\varphi$ is  surjective because for each $u^{-1}L$ in $R$, the state
  $p\in Q$ such that $i\edge{u}p$ is special.
  
We verify Condition 3, that is, $p\edge{w}q$ is a special path of $\A$ if and only if
$\varphi(p)\edge{w}\varphi(q)$ is an edge of $\A_c(L)$.
Indeed, if $p\edge{w}q$ is a special path, let $i\edge{u}p$
be a path in $\A$. Then $u^{-1}L\edge{w}(uw)^{-1}L$
is an edge of $\A_c$ since otherwise the path $p\edge{w}q$
would not be special. Conversely, if $\varphi(p)\edge{w}\varphi(q)$ is a special path of $\A_c(L)$, then it is an edge;
let $u$ be such that there is a path $i\edge{u}p$ if $\A$; then, by definition of edges in $\mathcal A_c(L)$, $\varphi(p)=u^{-1}L$ and $\varphi(q)=(uw)^{-1}L$; thus there is a 
path $i\edge{uw}q$ and finally, since the automaton is deterministic, a path  $p\edge{w}q$; it must be special, since $u^{-1}L
\edge{w} (uw)^{-1}L$ is an edge of $\A_c$.
  
Thus $\varphi$ is a reduction for $\A$ onto $\A_c(L)$.

We prove now uniqueness: let $\psi$ be some reduction from $\A$ onto $\A_c$.
Let $q=i\cdot u$ (that is, the unique state $q$ such that there is a path from $i$ to $q$ with label $u$). Then
$(Q,q,T)$ recognizes $u^{-1}L$. 
If $\psi(q)=k$, then it is easily verified that $\psi$ is a reduction from $(Q,q,T)$ onto $(R,k,S)$.
Hence by Proposition \ref{same}, these two automata recognize the same language.
But, since the states of $\mathcal A_c(u)$ are distinct residuals, $u^{-1}L$ is the unique state of $\A_c$ such that $(R,u^{-1}L,S)$ recognizes $u^{-1}L$.
Thus we must have $k=u^{-1}L$.
%
%
%
%
%
%
\end{proof}
Since all the states of the compact automaton $\A_c(L)$
are special, its number of states is at most the number
of special states of any compact automaton $\A$ recognizing $L$.
We have therefore the following statement which justifies
the name of minimal compact automaton for $\A_c(L)$.
\begin{corollary}\label{corollaryMinimalCompact}
  The compact automaton $\A_c(L)$ is, for every recognizable language $L$,
  the unique compact automaton with the minimal number of states
  which recognizes $L$.
\end{corollary}

Let $\A=(Q,i,T)$ be a trim compact deterministic automaton.
Let $q\in Q$ be a non-special state. Let $q\edge{v}r$
be the unique edge going out of $q$. Then $q\neq r$: indeed, if $q=r$, then $q$ is not co-accessible, since $q$ is not terminal (being not special), and there is no other outgoing edge than the loop $q\edge{v}q$; this contradicts that $\A$ is trim.
Since $i$ is special, we have also $q\neq i$.
Consider the compact automaton $\A'=(Q\setminus\{q\},i,T)$
with set of edges
\begin{enumerate}
\item[(i)] the edges $(p,w,r)$ of $\A$ with $p,r\ne q$,
\item[(ii)] the edges $(p,uv,r)$ for every edge $(p,u,q)$
  of $\A$ 
  \end{enumerate}
The identity map of $Q_s$ is a reduction
from $\A$ onto $\A'$, called an \emph{elementary reduction}.

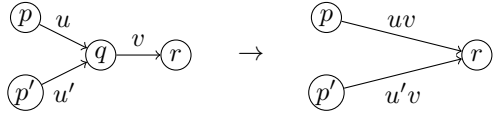
\begin{figure}
\centering
      \tikzset{node/.style={circle,draw,minimum size=0.4cm,inner sep=0.4pt}}
  \tikzset{box/.style={draw,minimum size=0.4cm,inner sep=0pt}}
  \tikzstyle{loop right}=[in=-40,out=40,loop]
  \begin{tikzpicture}
    \node[node](p)at(0,.5){$p$};\node[node](p')at(0,-.5){$p'$};
    \node[node](q)at(1,0){$q$};
    \node[node](r)at(2,0){$r$};
    \node at(3,0){$\rightarrow$};
    \node[node](p2)at(4,.5){$p$};
    \node[node](p'2)at(4,-.5){$p'$};
    \node[node](r2)at(6,0){$r$};

    \draw[above,->](p)edge node{$u$}(q);
    \draw[below,->](p')edge node{$u'$}(q);
    \draw[above,->](q)edge node{$v$}(r);
    \draw[above,->](p2)edge node{$uv$}(r2);
    \draw[below,->](p'2)edge node{$u'v$}(r2);
    \end{tikzpicture}
    \caption{An elementary reduction.}\label{figureElementary}
  \end{figure}

\begin{proposition}\label{elementary}
The minimal compact automaton $\A_c(L)$ is obtained from
$\A(L)$ by a sequence of elementary reductions.
\end{proposition}
\begin{proof}
Consider the deterministic compact automata recognizing $L$, having the following property, denoted by (R): for each state $q$, reachable from the initial state by a path labelled $u$, define the (well defined) mapping $\varphi_\A$ from the set of states into the set of residuals of $L$ by $q\mapsto u^{-1}L$; then this mapping is injective.

The minimal automaton $\A(L)$ has property $(R)$. We claim that property (R) is preserved by each elementary reduction. If an 
automaton has property (R) and has only special states, then it must be the minimal compact automaton. This proves the 
proposition.

We prove the claim. Let $\A$ and $\A'$ be as above. Then, as is easily verified, one has $\varphi_{\A'}=\varphi_\A|(Q\setminus \{q\})$, and more precisely, if $p\neq q$ is reachable by $u$ from $i$ in $\A$, then it is reachable from $i$ by $u$ in $\A'$ and therefore $\varphi_\A(p)=u^{-1}L=\varphi_{\A'}(p)$.
%
  \end{proof}
\begin{example}
  Let $\A$ be the deterministic
  automaton represented in Figure~\ref{figureS(abc)}.

The special states are $0,1,3,7$. There is a reduction from
this automaton to the compact automaton of Figure~\ref{figureCompact2}.

Two elementary reductions,
 suppressing $5,6$
give the automaton of Figure~\ref{figureCompact3}.
\begin{figure}[hbt]
    \centering
      \tikzset{node/.style={circle,draw,minimum size=0.4cm,inner sep=0.4pt}}
  \tikzset{box/.style={draw,minimum size=0.4cm,inner sep=0pt}}
  \tikzstyle{loop right}=[in=-40,out=40,loop]
  \begin{tikzpicture}
    \node[node,double](1)at(0,0){$0$};\draw[->](-.5,0)--node{}(1);
    \node[node,double](a)at(1.5,0){$1$};
    \node[node](ab)at(3,0){$2$};
    \node[node,double](aba)at(4.5,0){$3$};
    \node[node](abac)at(6,0){$4$};
    \node[node,double](abacaba)at(7.5,0){$7$};

    \draw[->,above](1)edge node{$a$}(a);
    \draw[->,above](a)edge node{$b$}(ab);
    \draw[->,above](ab)edge node{$a$}(aba);
    \draw[->,above](aba)edge node{$c$}(abac);
    \draw[->,above](abac)edge node{$aba$}(abacaba);

    \draw[bend left=40,above,->](1)edge node{$b$}(ab);
    \draw[bend left=40,above,->](a)edge node{$c$}(abac);
    \draw[bend right=40,above,->](1)edge node{$c$}(abac);
    \end{tikzpicture}
  \caption{Suppresion of $5,6$.}\label{figureCompact3}
\end{figure}
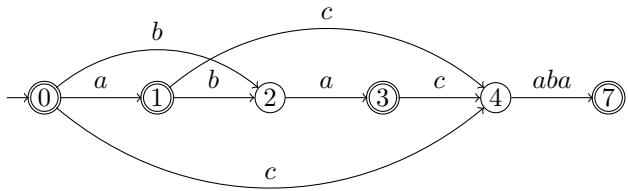  
  The suppression of $4$ gives then the compact automaton of Figure~\ref{figureCompact4}.
  \begin{figure}[hbt]
    \centering
      \tikzset{node/.style={circle,draw,minimum size=0.4cm,inner sep=0.4pt}}
  \tikzset{box/.style={draw,minimum size=0.4cm,inner sep=0pt}}
  \tikzstyle{loop right}=[in=-40,out=40,loop]
  \begin{tikzpicture}
    \node[node,double](1)at(0,0){$0$};\draw[->](-.5,0)--node{}(1);
    \node[node,double](a)at(1.5,0){$1$};
    \node[node](ab)at(3,0){$2$};
    \node[node,double](aba)at(4.5,0){$3$};
    
    \node[node,double](abacaba)at(6,0){$7$};

    \draw[->,above](1)edge node{$a$}(a);
    \draw[->,above](a)edge node{$b$}(ab);
    \draw[->,above](ab)edge node{$a$}(aba);
    \draw[->,above](aba)edge node{$caba$}(abacaba);

    \draw[bend left=40,above,->](1)edge node{$b$}(ab);
    \draw[bend left=40,above,->](a)edge node{$caba$}(abacaba);
    \draw[bend right=40,above,->](1)edge node{$caba$}(abacaba);
    \end{tikzpicture}
  \caption{Suppresion of $4$.}\label{figureCompact4}
\end{figure}
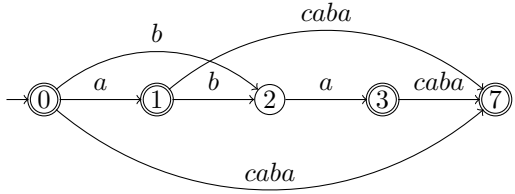 
  Finally, the suppression of $2$ gives the minimal compact automaton
  of Figure~\ref{figureCompact2}.
  \end{example}
  
 \section{Direct construction of the compact suffix automaton of $\Pal(u)$}\label{direct}
 
 In Theorem \ref{theoremPalSuf}, we have given some properties of the minimal automaton $\mathcal S(u)$ of the set of suffixes of $\Pal(u)$. By  
Proposition \ref{elementary}, we know how to transform this automaton into the minimal compact automaton of this set, which we denote $\mathcal S_c(u)$. 

In the present section, we construct directly this compact automaton.
One reason to proceed directly from $u$  to $S_c(u)$ is that
the number of states of $\mathcal S_c(u)$ is $1+$ the length of $u$ (by Lemma \ref{states} below), while
the number of states of $\mathcal S(u)$ (which is $1+$ the length of $Pal(u)$ by Theorem \ref{theoremPalSuf}) can be exponential in $|u|$ (for example, if $u=(ab)^n$,
the length of $Pal(u)$ is $F_{2n+3}-2$, which grows exponentially with $n$).

\begin{theorem}\label{Construction} The automaton $\Suf_c(u)$ is completely characterized as follows:
the states are the prefixes of $u$, all terminal, and $1$ is the initial state.
For each factorization $u=xyaz$, where $a$ is a letter and $x,y,z$ are words, with $y$ $a$-free, there is a  transition
$x \to xya$, 
labelled $\Pal(xy)^{-1}\Pal(xya)$.
\end{theorem}

Recall from Section 
\ref{sectionCompactAutomata} that the states of the automaton $\Suf_c(u)$ are the special residuals of $L$, the set of suffixes of $\Pal(u)$. By Theorem \ref{theoremPalSuf}, the nonempty residuals 
of $L$ are the $p^{-1}L$ where $p$ is a prefix of $\Pal(u)$. Clearly, the mapping $p\mapsto p^{-1}L$ is a bijection from 
the set of prefixes of $\Pal(u)$ onto the set of nonempty residuals of $L$. 
 
\begin{lemma}\label{states} The set of states $\Suf_c(u)$ is naturally in bijection with the set of palindromic prefixes of $\Pal(u)$. This set is naturally in bijection with the set of prefixes of $u$; with this identification, the initial state is $1$ and all states are terminal.
\end{lemma}
  
\begin{proof} The second bijection maps a prefix $p$ of $u$ onto the prefix $\Pal(p)$ of $\Pal(u)$ (see \cite {Justin2005} p. 209).

Let $p$ be a prefix of $\Pal(u)$. According to the definition of special residuals in Section \ref{sectionCompactAutomata}, $p^{-1}L$ is special if and only if (i) either 
$p=1$, or (ii) $p\in L$, or (iii) if there are two words in $p^{-1}L$ beginning by different letters. 

Let $p^{-1}L$ be a special residual of $L$. We show that in all three cases, $p$ is a palindromic prefix of 
$\Pal(u)$.

In case (i), $p=1$ which is clearly a palindromic prefix of $\Pal(u)$. In case (ii), $p$ is a suffix of $\Pal(u)$, and being also a 
prefix, it is a palindromic prefix of $\Pal(u)$. In case (iii), let $s,t$ be the two words, with $s=as'$, $t=bt'$, for distinct letters 
$a,b$; then $pas',pbt'$ are in $L$, hence $p$ is a right special factor of $\Pal(u)$; then $\tilde p$ is a left special factor of $\Pal(u)$, 
thus by Corollary \ref{corollaryLS}, $\tilde p$ is a prefix of $\Pal(u)$ and therefore $p$ is a palindromic prefix of $\Pal(u)$. 

Conversely, if $p$ is a palindromic prefix of $\Pal(u)$, then $p$ is also a suffix of $\Pal(u)$, hence $p\in L$ and $p^{-1}L$ is special residual of $L$, by case (ii).
\end{proof}
 
We use another formula of Justin, see \cite{Justin2005} p. 209.
Let $x \in A$. If $u$ is $x$-free then $Pal(ux) = Pal(u)xPal(u)$. If on the other hand  $x$ occurs in $u$, write $u = u_1xu_2$ with $u_2$ 
$x$-free. Then
\begin{equation}\label{Justin-bis}
\Pal(ux) = \Pal(u)\Pal(u_1)^{-1}\Pal(u).
\end{equation}
 
The recursive definition of $\mathcal S_c(u)$ is is explained in the following result.
 
\begin{proposition}\label{construction} Let $u\in A^*, x\in A$. Define $u=hu_2$, where $u_2$ is the longest $x$-free suffix of $u$. The automaton $\Suf_c(u)$ having as set of states the set of prefixes of $u$, as stated in Lemma \ref{states}, construct an automaton $\mathcal S$ as follows:
\begin{itemize}
\item add to $\Suf_c(u)$ the new state $ux$, which is terminal;
\item for each prefix $p$ of $u_2$, add an edge from the state $hp$ of $\Suf_c(u)$ to the new state $ux$, labelled $\Pal(u)^{-1}\Pal(ux)$.

Then $\mathcal S=\mathcal S_c(u)$.
\end{itemize}
\end{proposition} 

Note that $\Pal(u)$ is a prefix of $\Pal(ux)$, so that $\Pal(u)^{-1}\Pal(ux)\in A^*$. Moreover, $hp$ is a prefix of $u$, hence is a state of $\Suf_c(x)$.

\begin{figure}[hbt]
  \centering
  \tikzset{node/.style={circle,draw,minimum size=0.4cm,inner sep=0.4pt}}
  \tikzset{box/.style={draw,minimum size=0.4cm,inner sep=0pt}}
  \tikzstyle{loop right}=[in=-40,out=40,loop]
  \begin{tikzpicture}
\node[node,double](1)at(0,0){$1$};\draw[->](-.5,0)--node{}(1);
    \node[node,double](a)at(2,0){$a$};
    \node[node,double](aba)at(4,0){$ab$};
    \node[node,double](abacaba)at(6,0){$abc$};
    \node[node,double](abacabaa)at(8,0){$abca$};

    \draw[->,above](1)edge node{}(a);
    \draw[->,above](a)edge node{}(aba);
    \draw[->,above](aba)edge node{$$}(abacaba);
    \draw[->,above](abacaba)edge node{$abacaba$}(abacabaa);
    
    \draw[bend left=40,above,->](a)edge node{$abacaba$}(abacabaa);
    \draw[bend right=40,below,->](aba)edge node{$abacaba$}(abacabaa);
    \end{tikzpicture}
\caption{From $\mathcal S_c(abc)$ to $\mathcal S_c(abca)$}\label{figureCompact3}
\end{figure}
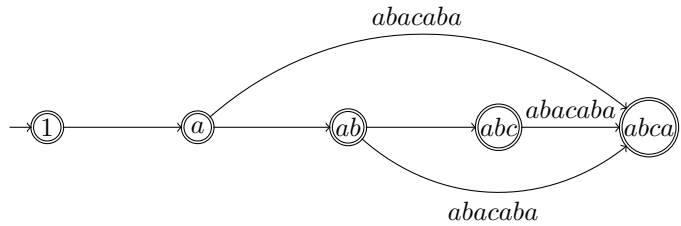

Figure \ref{figureCompact3} illustrates the construction in Proposition \ref{construction}: the construction from $\mathcal S_c(abc)$ in Figure \ref{figureCompact2} to $\mathcal S_c(abca)$; here, $u=abc$, $h=a, u_2=bc$ and only the new edges are drawn. Note that $\Pal(abc)^{-1}\Pal(abca)=abacaba$.

\begin{lemma}\label{auto} 
Each word 
recognized by the automaton $\mathcal S$ of Proposition \ref{construction} is a 
suffix of $Pal(ux)$.
\end{lemma}


\begin{proof} 
Let $w$ be a word recognized by $\mathcal S$, that is the label of some path in $\mathcal S$. If this path does not end at $ux$, then by the construction, it is a path in $\Suf_c(u)$; 
hence $w$ is a suffix of $\Pal(u)$, and since $\Pal(u)$ is a suffix of $\Pal(ux)$ (the former is a prefix of the latter, and both words are palindromes), $w$ is a suffix of $\Pal(ux)$. If this path ends at $ux$, 
then its last edge is one of the new edges; hence $w=s\Pal(u)^{-1}\Pal(ux)$, where $s$ is a suffix 
of $\Pal(u)$. Suppose first that $u$ is $x$-free; then by Justin's result recalled above, $\Pal(ux)=\Pal(u)x\Pal(u)$ and $w=sx\Pal(u)
$ is a suffix of $\Pal(u)x\Pal(u)=Pal(ux)$. Suppose now that $u$ is not $x$-free; then $u=u_1xu_2$, $u_2$ is $x$-free and $
\Pal(ux)=\Pal(u)\Pal(u_1)^{-1}\Pal(u)$; moreover, $w=s\Pal(u)^{-1}\Pal(u)\Pal(u_1)^{-1}\Pal(u)=s\Pal(u_1)^{-1}\Pal(u)$; since $s$ is 
a suffix of $\Pal(u)$ and since $\Pal(u_1)^{-1}\Pal(u)$ is in $A^*$, we see that $w$ is a suffix of $\Pal(u)
\Pal(u_1)^{-1}\Pal(u)=\Pal(ux)$. 

\end{proof}

\begin{lemma}\label{keeping} For every prefix of $p$ of $u$, $
\Suf_c(p)$ is obtained from $\Suf_c(u)$ by keeping in the latter only the states which are prefixes of $p$.

The 
number of paths in $\Suf_c(u)$ from the initial state to the state $u$ is equal to $|\Pal(u)|-|\Pal(u^-)|$ if $u$ is nonempty (where $x^-$ 
denotes the word $x$ with the last letter removed), and it is $1$ if $u$ is empty.
\end{lemma}

\begin{proof}
We prove the lemma by induction on the length of $u$.
For $u=1$, it is immediate.

Suppose now that $u\in A^*$ and $x\in A$. We prove the assertions for $ux$, admitting them for shorter words.

Take the notation $\mathcal S$ and $u=hu_2$ in Proposition \ref{construction}. We know by Lemma \ref{auto} that each word recognized by $\mathcal S$ is recognized by $\Suf_c(ux)$. We prove now the converse, by a
counting argument, using the induction hypothesis. Let $n_w$ denote the number of words recognized by $\Suf_c(w)$. Then $n_w$ is equal 
to the number of suffixes of $\Pal(w)$, hence is equal to
to $1+$ the length of $\Pal(w)
$. Hence $n_{ux}-n_u=|\Pal(ux)|-|\Pal(u)|$. Now let $n$ be the number of words recognized by $\mathcal S$. By construction of 
the automaton $\mathcal S$, each word recognized by it, and not recognized by $\Suf_c(u)$, is of the form 
$w=s\Pal(u)^{-1}\Pal(ux)$, where $s$ is the label of some path in $\Suf_c(u)$ which starts at $1$ and ends at $hp$ for some prefix 
$p$ of $u_2$. By induction, the number of such words $s$ is equal to $|\Pal(hp)|-|\Pal((hp)^-)|$ if $hp$ is nonempty, and 1 if $hp$ is empty. 
Since the corresponding sum is telescoping, it follows that the number $n-n_u$ of possible words $s$ is equal to $|\Pal(u)|-|\Pal(h^-)|$ if $h$ is nonempty, and to 
$1+|\Pal(u)|$ if $h$ is empty. If $u$ is not $x$-free, then $h=u_1x$ is nonempty, $h^-=u_1$ and $\Pal(ux)=\Pal(u)
\Pal(u_1)^{-1}\Pal(u)$, so that $n_{ux}-n_u=|\Pal(ux)|-|\Pal(u)|=|\Pal(u_1)^{-1}\Pal(u)|=|\Pal(u)|-|u_1|=|\Pal(u)|-|h^-|=n-n_u$. If $u$ is $x$-free, then 
$h=1$, $\Pal(ux)=\Pal(u)x\Pal(u)$, and $n_{ux}-n_u=|\Pal(ux)|-|\Pal(u)|=|x\Pal(u)|=1+|\Pal(u)|=n-n_u$. Thus in both cases, $n_{ux}=n$, which implies that $\mathcal S_c(ux)$ and $\mathcal S$ both recognize the language of suffixes of $\Pal(ux)$; since both automata have the same number of states and the first is minimal, they are isomorphic; but since both have a unique longest path of the same length, they are equal.

The two assertions of the lemma now clearly follow for $ux$.
\end{proof}

\begin{proofof}{of Proposition~\ref{construction}}
It follows from the proof of Lemma \ref{keeping}.
\end{proofof}

\begin{proofof}{Theorem \ref{Construction}}
The theorem follows by a straightforward induction from Proposition \ref{construction}.
\end{proofof}

Theorem \ref{Construction} has the following corollary, which could also be proved using (iv) in Section \ref{sectionSuffixAutomaton} and Proposition \ref{elementary}.

\begin{corollary}\label{only} In the graph $\mathcal S_c(u)$, the label of an edge depends only on the final state of the edge.
\end{corollary}

\begin{proof}
Indeed, the label of each transition $v\to w$ is $\Pal(w^-)^{-1}\Pal(w)$.
\end{proof}

\begin{corollary}\label{count} Let $u\in A^*$, of length $n$, and for any $a\in A$, denote by $p_a(u)$ the position of the rightmost occurrence of 
$a$ in the word $u$, with $p_a(u)=0$ when $u$ is $a$-free. Then the number of states in $\mathcal S_c(u)$ is $n+1$ and the 
number of transitions is $\sum_{a\in A}p_a(u)$. 
\end{corollary}

\begin{proof} By Theorem \ref{Construction}, the number of states is the number of prefixes of $u$, thus it is $n+1$. The number 
$t$ of transitions is equal to $t=\sum_{a\in A} t_a$, where
$t_a$ is the number of factorizations $u=xyaz$, $x,y,z\in A^*$, $y$ $a$-free. We have $t_a=0$ if $u$ is $a$-free.

Suppose that $u$ contains $a$. Denote $I_a=A^*aA^*$, the set of words containing letter $a$; clearly, each word $v$ in $I_a$ has 
a unique factorization $v=yaz$, $y,z\in A^*$, $y$ $a$-free. Hence $t_a$ is equal to the number of factorizations $u=xv$, $$ (*) \,\, x\in 
A^*,v\in I_a.$$ 
If we factorize $u=u_1au_2$, where $u_2$ is the longest $a$-free suffix of $u$, then $|u_1|+1=p_a(u)$, and a factorization $u=xv$ 
satisfies $(*)$ if and only if $x$ is a prefix of $u_1$. Hence the number of such factorizations is $|u_1|+1=p_a(u)$.
\end{proof}

For a binary alphabet, Theorem \ref{Construction} was obtained, in an another but equivalent form, by Epifanio, Mignosi, Shallit 
and Venturini 
\cite{EpifanioMignosiShallitVenturini2007} (see also \cite{BugeaudReutenauer2022}, especially Figure 1). Similarly for Corollary \ref{count}, see \cite{EpifanioMignosiShallitVenturini2007} Proposition 1.

\section{Further comments}

Following \cite{EpifanioMignosiShallitVenturini2007}, we obtain for any word $u$ on 
any alphabet, a directed graph that {\it counts from 0 to $n=|\Pal(u)|$} in the following sense: replace in $\mathcal S_c(u)$ each 
label by 
its length; then one obtains a directed graph, with the initial vertex $1$, such that for each $k=0,\ldots, n$, there is a unique path,
starting from the initial vertex, whose label is $k$ (here, labels of paths here additive). This follows clearly since there is a unique 
suffix of $\Pal(u)$ of each length $k=0,\ldots, n$. For $u$ on a binary alphabet, Epifanio et al. call this graph a {\it Sturmian 
graph}.

As open problem, we mention that Theorem \ref{construction} has certainly an interpretation as a factorization result: each suffix of 
$\Pal(u)$ as a certain factorization as a product of the words which are the label of the edges of the automaton; these words are all 
of the form $\Pal(p^-) \Pal(p)$, $p$ a proper prefix of $u$. In the binary alphabet case, this is known: the factorization is related to 
the Ostrowski lazy factorization (see \cite{EpifanioFrougnyGabrieleMignosiShallit2012}), and to the factorization theorem of Anna 
Frid \cite{Frid2018} Corollary 1, which following
\cite{BugeaudReutenauer2022} Section 9, may be stated as follows: for $u$ on a binary alphabet, of length $n$, each suffix $s$ of 
$\Pal(u)$, of length $\ell$, has a unique factorization $\prod_{1\leq i\leq n}L_0^{d_1}L_1^{d_2}\cdots L_{n-1}^{d_n}$, where 
$L_i=\Pal(p_{i}^-)^{-1}\Pal(p_{i+1})$, $p_i$ the prefix of length $i$ of $u$, and where $\ell =\sum_{1\leq i\leq n}d_iq_{i-1}$ is the 
lazy Ostrowski representation of $\ell$, with $q_j=|L_j|$.

\bibliographystyle{plain}
\bibliography{pal}

\end{document}